\documentclass[review]{elsarticle}

\usepackage[T1]{fontenc}
\usepackage{hyperref}
\usepackage{amsmath}
\usepackage{amssymb}
\usepackage{graphicx}
\usepackage{color}
\newcommand{\F}{ \textbf{F} }
\newcommand{\C}{ \textbf{C} }
\newcommand{\p}{ \mathcal{P} }
\newcommand{\A}{ \mathcal{A} }
\newcommand{\z}{\phantom{0}}
\newcommand{\I}{ \boldsymbol{i} }
\newcommand{\q}{ \boldsymbol{q} }
\newcommand{\x}{ \boldsymbol{x} }
\newcommand{\y}{ \textbf{y} }
\newcommand{\f}{ \textbf{f} }
\newcommand{\U}{ \textbf{u} }
\newcommand{\B}{ \textbf{b} }

\newtheorem{prop}{Proposition}
\newtheorem{remark}{Remark}
\newproof{proof}{Proof}

\def\R{\color{black}} 

\numberwithin{equation}{section}
\begin{document}

\begin{frontmatter}

\title {A preconditioning strategy for linear systems arising
from \\ nonsymmetric schemes in isogeometric analysis}
\author{Mattia Tani}
\address{Universit\`a di Pavia, Dipartimento di Matematica ``F. Casorati'', 
via A. Ferrata 5, 27100 Pavia, Italy.}
\ead{mattia.tani@unipv.it}


\begin{abstract} 
In the context of isogeometric analysis, we consider two discretization approaches
{\R that make} the resulting stiffness matrix nonsymmetric even if the differential operator
is self-adjoint. These are the collocation method and the recently-developed
weighted quadrature for the Galerkin discretization. In this paper, we are interested
in the solution of the linear systems arising from the discretization of
the Poisson problem using one of these approaches. In [SIAM J. Sci. Comput.
38(6) (2016) pp. A3644--A3671], a well-established direct solver for linear
systems with tensor structure was used as a preconditioner in the context of
Galerkin isogeometric analysis, yielding promising results. In particular, this
preconditioner is robust with respect to the mesh size $h$ and the spline degree
$p$. In the present work, we discuss how a similar approach can applied to the
considered nonsymmetric linear systems. The efficiency of the proposed preconditioning
strategy is assessed with numerical experiments on two-dimensional
and three-dimensional problems.
\end{abstract}


\end{frontmatter}

\section{Introduction}

Isogeometric analysis (IGA) is a method to numerically solve PDEs, which
was first proposed in \cite{Hughes2005} as an attempt to create a link between the Computer
Aided Design (CAD) and finite elements analysis. In IGA, the functions that are
used to approximate the solution of the considered PDE are B-splines and their
generalizations (e.g. NURBS). Differently from the $C^0$ piecewise polynomials
used in the finite element method, such functions can have a high regularity.
In particular, in the so-called $k$-method, which is the primary focus of this
work, $C^{p-1}$ piecewise polynomials (or rational functions in case of NURBS) of
degree $p$ are considered. One important implication of the higher regularity of
the basis functions is that they yield a higher accuracy per degree-of-freedom
when compared to standard finite elements. On the other hand, they also yield
a higher computational cost, at least in currently available software.

The computational core of any numerical solver for a linear PDE is typically
represented by a linear algebra phase, where a linear system is assembled and
then solved. In the IGA Galerkin method, the first task is considered particularly
challenging from the computational point of view. Indeed, the common
approach to assemble the system matrix, inherited by finite element codes, is to
perform a loop over the elements. On each element the local integrals are approximated
e.g. by Gaussian quadrature and then added to the global matrix.
The computational cost of the overall process, in case of maximum regularity
splines of degree $p$, is $O(N p^{3d})$ FLOPS, where $N$ is the total number of degrees
of freedom and $d$ is the dimension of the physical domain. Clearly, this cost
makes the use {\R of} high-order splines unfeasible for practical problems.

One possible strategy to work around the problem is to abandon the Galerkin
approach in favor of a collocation method \cite{Auricchio2010,Auricchio2012a,Schillinger2013}. 
Here the approximate solution is found by enforcing that the differential equation is satisfied on a given
finite set of points. One of the appealing feature of this approach is that no
quadrature is needed, and the formation of the system matrix requires only the
evaluation of the basis functions (and their derivatives) at the collocation points.

Within the Galerkin framework, in the recent years a significant number of
papers has focused on strategies that mitigate the cost of numerical quadrature,
see e.g. \cite{Hughes2010,Auricchio2012,Schillinger2014,Mantzaflaris2014,Adam2015,Johannessen2016,Mantzaflaris2016}.

A very recent strategy was presented in \cite{Calabro2017}, where the idea of weighted
quadrature (WQ) is introduced. In this approach, the matrix system is assembled
row by row, rather than element by element. For each matrix row, the test
function is incorporated into the integral measure. In this way, a row-dependent
quadrature rule is derived by imposing some exactness conditions. If this strategy
is combined with sum-factorization to exploit the tensor structure of the
basis functions, the final cost of matrix assembly is $O(N p^{d+1})$.

A common feature of the mentioned approaches, collocation and WQ Galerkin,
is that the resulting system matrix is in general nonsymmetric, even when the
underlying differential operator is self-adjoint. This lack of symmetry might
be considered an unfavorable property when it comes to solving the linear system.
This is true in particular for iterative solvers, which are the focus of this
work. As far as we know, the only preconditioners that have been employed
in the literature to tackle IGA collocation systems are: a simple yet effective
incomplete LU factorization \cite{Schillinger2013,Anitescu2015}, an Overlapping Schwarz method \cite{Veiga2014a}, and
a multi-iterative preconditioner based on the analysis of the spectral symbol of
the system matrix \cite{Donatelli2015a}.

The Fast Diagonalization (FD) method \cite{Lynch1964,Deville2002} is a well-established direct
solver for linear systems having a special tensor structure. Recently, it has
been shown \cite{Sangalli2016} that the {\R FD} method is a very efficient and robust preconditioner
for symmetric systems arising from the isogeometric discretization
of elliptic problems. This approach fully exploits the tensor structure that is
peculiar to spline basis functions.

The aim of the present work is to discuss the application of a preconditioning
strategy based on the same ideas in the nonsymmetric setting. In particular,
we discuss the spectrum of the preconditioned and unpreconditioned stiffness
matrix arising from numerical quadrature of the stiffness matrix. We prove that under a reasonable assumption on the quadrature approach, which is numerically shown to be valid in case of WQ, the spectrum of such approximated matrices {\R converges} to the spectrum of the exact matrices. This means in particular
that the proposed preconditioner, which is robust for the exact Galerkin
matrix, is also robust for the approximated one. Concerning the application of
the preconditioner, we present and discuss two popular algorithms, the Bartels-
Stewart (BS) method {\R \cite{Bartels1972}} and the nonsymmetric version of the FD method. The
discussion reveals that the latter is more suited for the particular features of
IGA problems. We report numerical experiments using the FD method as preconditioner,
which {\R show} that the proposed approach has the same appealing
features as in the symmetric setting. In particular, the number of iterations on
the preconditioned system is independent {\R of} the mesh size $h$ and {\R of} the spline
degree $p$. Moreover, the FD method in the 3D case is so fast that the effort
required for the application of the preconditioner is negligible with respect to
that of the matrix-vector product which is the fundamental step of any iterative
method.

The outline of the paper is as follows. In Section \ref{sec:preliminaries} we set the notation and
introduce the isogeometric collocation method and the weighted quadrature approach.
In Section \ref{sec:preconditioner} we define the preconditioner {\R and discuss its robustness.}
In Section \ref{sec:methods} we describe the BS and FD methods, and discuss their
computational cost. Numerical experiments with the FD preconditioner are reported
in Section 5. Finally, in Section \ref{sec:conclusions} we draw the conclusions.

\section{Preliminaries} \label{sec:preliminaries}

\subsection{The problem}

As a model problem, we consider the following Poisson problem with Dirichlet boundary
conditions:
\begin{equation}\label{poisson}
\left\{
\begin{aligned}
-\text{div} ( {K} (\mathbf{x}) \nabla \mathrm{u} (\mathbf{x}) ) = f (\mathbf{x}) &\quad \text{on} \quad \Omega, \\
\mathrm{u} = 0 & \quad \text{on} \quad \partial \Omega,
\end{aligned}
\right.
\end{equation}
where $\Omega$ is a domain of $\mathbb{R}^d$ ($d=2,3$), {\R and $K(\mathbf{x})$ is a $d \times d$ symmetric positive definite matrix for each $\mathbf{x} \in \Omega$.}

{\R Given two positive integers $p$ and $m$,} for each direction $l=1,\ldots,d$ we consider the open knot vector $\Xi_l = {\R\left\{\xi_1^{(l)}, \ldots, \xi^{(l)}_{m+p+1} \right\}}$, with
{\R
$$ 0 = \xi_1 =\ldots=\xi_{p+1} < \xi_{p+2} \le \ldots \le \xi_{m} < \xi_{m+1}=\ldots=\xi_{m+p+1} = 1. $$
where repeated knots are allowed, up to multiplicity $p$.} We denote with $\hat{B}_{l,i}$, $i = 0, 1, \ldots, m - 1$, the univariate spline basis functions of degree $p$ corresponding to the knot vector $\Xi_l$, $l = 1, \ldots,d$ (we refer to \cite[Chapter 2]{Cottrell2009} for the details). These are piecewise polynomials of degree $p$ which are $C^{p-r}$ continuous at nodes with multiplicity $r$.
While we are mainly interested in the case of maximum regularity $p-1$ (i.e. all knots have multiplicity 1 except the first and the last), this is not restrictive for the presented approach and any continuity is allowed.

Multivariate B-splines are defined by tensor product of univariate B-splines. Given a multi-index $\I = (i_1, \ldots, i_d)$, we define
\begin{equation} \label{eq:multivariatesplines} \hat{B}_{\I}(\xi) = \hat{B}_{i_1,1}\left(\xi_1\right) \ldots \hat{B}_{i_d,d} \left(\xi_d\right), \end{equation}
where $\xi = \left(\xi_1,\ldots,\xi_d\right)$. {\R To avoid proliferation of notation, we assume that all the univariate spline spaces have the same order and the same dimension. We emphasize, however, that this is not restrictive for the approaches described in the present work.}

For simplicity, we restrict to the case when $\Omega$ is given by a single-patch spline parametrization, i.e. we assume that $\Omega = \F (\left[0,1\right]^d)$, where $\F$ has the form
$$ \F(\xi) = \sum_{\I} \C_{\I} \hat{B}_{\I}(\xi), \qquad \C_{\I} \in \mathbb{R}^d $$

We consider the function search space for problem \eqref{poisson}:
\begin{equation} \label{eq:space} V_h = \left\{ \hat{B}_{\I} \circ \F^{-1} \left.\right| \I = (i_1,\ldots,i_d), \; 1 \leq i_l \leq m-2, \; 1\leq l \leq d \right\}. \end{equation}
Note that the homogeneous Dirichlet boundary conditions are incorporated in the definition of $V_h$. If we define $n := m - 2$, the number of degrees of freedom is given by $N = \mbox{dim} (V_h) = n^d$. In the following, with abuse of notation we will identify the multi-index $\I = (i_1,\ldots,i_d)$ with the scalar index ${\R i = 1+ \sum_{l=1}^d n^{l-1} (i_l - 1)}$.

\subsection{The collocation method}

We consider the collocation points $\tau_{\I}= \left(\tau_{1,i_1},\ldots,\tau_{d,i_d}\right)$ where $\tau_{l,i_l} \in \left[0,1\right]$, $i_l = 1,\ldots,n$, $l = 1,\ldots,d$ are suitably-chosen scalar points. The most common choice {\R is} the Greville abscissae, i.e.
\begin{equation} \label{eq:greville} {\R \tau_{l,i} = \frac{1}{p} \sum_{j=1}^p \xi_{i + j + 1}^{(l)}, \qquad i = 1,\ldots,n, }\end{equation}
but other choices are possible, see e.g. \cite{DeBoor1973,Demko1985,Auricchio2010,Anitescu2015,Gomez2016,Montardini2016}. If, for any $u_h \in V_h$, we define
$$ \mathcal{L} u_h := - \mbox{div}\left(K \nabla u_h\right), $$
then the approximate solution $u_h^C \in V_h$ is defined by the conditions
\begin{equation} \label{eq:collocation} \left(\mathcal{L} u_h^C\right) \left(\F \left(\tau_{\I}\right)\right) = f\left(\F \left(\tau_{\I}\right)\right), \qquad \I = 1,\ldots,N. \end{equation}
If $\U_C \in \mathbb{R}^N$ represents the vector of coefficients of $u_h^C$ with respect to the spline basis \eqref{eq:multivariatesplines}, then the conditions \eqref{eq:collocation} can be written as a linear system, that is
\begin{equation} \label{eq:collocationsystem} \mathcal{A}_C \U_C = \f_C, \end{equation}
where ${\R \f_C = \left[f(\F(\tau_1)), \ldots, f(\F(\tau_N))\right]^T} $. We emphasize that the matrix system $\mathcal{A}_C$ is not symmetric.

The expression of $\mathcal{A}_C$ in the general case is complicated, and we do not report it. However, it greatly simplifies when $\F$ is the identity function (thus in particular $\Omega = [0, 1]^d$), and $K(\mathbf{x}) \equiv I_d$. For this special case we write an explicit
expression for $\mathcal{A}_C$, as it will be useful in the following.

For $d=2$, we have
\begin{equation} \label{eq:colloc2D} \mathcal{A}_C = K_2^C \otimes M_1^C + M_2^C \otimes K_1^C, \end{equation}
where $\otimes$ denotes the Kronecker product \cite[Section 1.3.6]{Golub2012}\textsc{}, and the Kronecker factors are
defined as:
\begin{equation} \label{eq:univariatecolloc} \left(M_l^C\right)_{ij} = \hat{B}_{l,j}\left(\tau_{l,i}\right), \qquad \left(K_l^C\right)_{ij} = - \hat{B}^{''}_{l,j}\left(\tau_{l,i}\right), \qquad l = 1,2, \quad i,j = 1,\ldots,n. \end{equation}
For $d=3$ we have instead
\begin{equation} \label{eq:colloc3D} \mathcal{A}_C = K_3^C \otimes M_2^C \otimes M_1^C + M_3^C \otimes K_2^C \otimes M_1^C + M_3^C \otimes M_2^C \otimes K_1^C, \end{equation}
with similar definition of the Kronecker factors.

\subsection{Weighted quadrature}
We consider the Galerkin projection of problem \eqref{poisson} on the space $V_h$. The {\R exact stiffness matrix, denoted by $\A_G$}, has the form
\begin{eqnarray*}
\left(\mathcal{A}_G\right)_{\I\boldsymbol{j}} & = & \int_{\left[0,1\right]^d} \left(\nabla \hat{B}_{\I}\left(\xi\right)\right)^T Q(\xi) \nabla \hat{B}_{\boldsymbol{j}}\left(\xi\right) \; d \xi \\
& = & \sum_{\alpha,\beta = 1}^d \int_{\left[0,1\right]^d} {\R c_{\alpha,\beta}}\left(\xi\right) \partial_\alpha \hat{B}_{\I}\left(\xi\right)
\partial_\beta \hat{B}_{\boldsymbol{j}}\left(\xi\right) \; d \xi, \qquad \I,\boldsymbol{j} = 1,\ldots,N. \end{eqnarray*}
Here matrix $Q(\xi)$ is given by
\begin{equation} {\R \label{eq:Q} Q(\xi) = \mbox{det}\left(J_{\F}(\xi)\right) \; J_{\F}\left(\xi\right)^{-1} K\left(\F(\xi)\right) J_{\F}\left(\xi\right)^{-T}. }\end{equation}
where $J_{\F}$ is the Jacobian of $\F$, and the scalar-valued functions ${\R c_{\alpha,\beta}}$, $\alpha,\beta = 1,\ldots,d$ are the entries of $Q$.

In practice, the integrals appearing in the entries of $\mathcal{A}_G$ have to be approximated by numerical quadrature. Following the weighted quadrature approach \cite{Calabro2017}, we incorporate the function $\partial_\alpha \hat{B}_{\I} $ in the quadrature weights to get
\begin{equation} \label{eq:WQapprox}
\left(\mathcal{A}_{wq}\right)_{\I\boldsymbol{j}} := \sum_{\alpha,\beta = 1}^d \sum_{\q} {\R c_{\alpha \beta}}(\x_{{\R \q}}) \; w_{\I \q}^{\alpha \beta} \; \partial_\beta \hat{B}_{\boldsymbol{j}} (\x_{\q}) \approx \left(\A_{G}\right)_{\I \boldsymbol{j}},
\end{equation}
for some suitably-chosen quadrature points $\x_{\q}$ and weights $w_{\I \q}^{\alpha \beta}$. In \cite{Calabro2017} the weights are chosen by imposing that quadrature approximations \eqref{eq:WQapprox} are exact when the ${\R c_{\alpha \beta}}$ are constant functions. Thus in this case we have $\mathcal{A}_{wq} = \mathcal{A}_G$. We emphasize that, even if $\mathcal{A}_G$ is symmetric, the approximation $\mathcal{A}_{wq}$ is not, in general. We refer to \cite{Calabro2017} for more details.

{\R The linear system obtained with the WQ approach has the form}
\begin{equation} \label{eq:WQsystem} 
\mathcal{A}_{wq} \U_{wq} = \f_{wq}.
\end{equation}
Here $\f_{wq} {\R \in \mathbb{R}^N}$ represents a suitable approximation of the Galerkin right-hand side, which can be obtained either by using again the WQ approach, or by some different quadrature approach (e.g. Gaussian quadrature). Similarly as before, a particularly interesting case occurs when $\F$ is the identity mapping and $K(\mathbf{x}) \equiv I_d$. Note that in this case $Q(\xi) \equiv I_d$ and the quadrature formulas \eqref{eq:WQapprox} are exact. Thus for $d=2$ we have
\begin{equation} \label{eq:wq2D} \mathcal{A}_{wq} = K_2^{wq} \otimes M_1^{wq} +M_2^{wq} \otimes K_1^{wq}, \end{equation}
where
\begin{equation} \label{eq:univariateWQ} 
\left(M_l^{wq}\right)_{ij} = \int_0^1 \hat{B}_{l,i}(\xi) \hat{B}_{l,j}(\xi) \; d \xi, \qquad 
\left(K_l^{wq}\right)_{ij} = \int_0^1 \hat{B}^{'}_{l,i}(\xi) \hat{B}^{'}_{l,j}(\xi) \; d \xi,
\end{equation}
for $l = 1,2$, $i,j = 1,\ldots,n$. Note in particular that $M_1^{wq}$, $M_2^{wq}$, $K_1^{wq}$, $K_2^{wq}$ are symmetric and positive definite. For $d=3$ we have instead
\begin{equation} \label{eq:wq3D} \mathcal{A}_{wq} = K_3^{wq} \otimes M_2^{wq} \otimes M_1^{wq} + M_3^{wq} \otimes K_2^{wq} \otimes M_1^{wq} + M_3^{wq} \otimes M_2^{wq} \otimes K_1^{wq}, \end{equation}
with similar definition of the Kronecker factors.

\section{The preconditioner} \label{sec:preconditioner}

We are interested in the iterative solution of the ``abstract'' linear system
\begin{equation}
\mathcal{A} \U = \f,
\end{equation}
which can represent either the collocation system \eqref{eq:collocationsystem} or the WQ Galerkin
system \eqref{eq:WQsystem}.

Since $\mathcal{A}$ is nonsymmetric, we consider a Krylov method for nonsymmetric systems, such as GMRES \cite{Saad1986} or BICGStab \cite{VanderVorst1992}. Of course, in order to have an efficient and robust method, we need a good preconditioner.

The quality of a preconditioner $\p$ depends on two features. The first one is that a Krylov method applied to the preconditioned system should converge
fast; it is typically required that the spectrum of $\p^{-1} \mathcal{A}$ is clustered. Even though in the nonsymmetric case this does not guarantee fast convergence from a theoretical point of view 
, for practical problems a clustered spectrum is typically associated with fast convergence. We say that a preconditioner is robust if the eigenvalues of $\p^{-1} \mathcal{A}$ are bounded away from 0 and infinity by constants that do not depend on the parameters of the problem.

In the spirit of the $k-$method, we are primarily interested in a preconditioner
which is robust with respect to both the spline degree $p$ and the mesh size
$h$. The second feature of a good preconditioner is that the cost of computing
the action of $\p^{-1}$ on vectors should be low, e.g. comparable with the cost of
computing matrix-vector products with $\mathcal{A}$.

{\R In the spirit of \cite{Sangalli2016}}, we choose as preconditioner for $\mathcal{A}$ the matrix $\p$ which represents the same differential problem \eqref{poisson}, discretized with the same approach as $\mathcal{A}$ {\R (collocation or WQ Galerkin)}, but where the domain is replaced with the unit $d-$dimensional cube $[0, 1]^d$. 
{\R 
We remark that a similar strategy was already considered for collocation in \cite{Donatelli2015a}, where $\p$ is referred to as the Parametric Laplacian matrix; in that work, however, the preconditioner is applied inexactly using a multi-iterative approach.

We now specialize the notation for $\p$ in the different settings, similarly as we did for the stiffness matrix $\A$ in the previous section. In the context on collocation, the preconditioner is denoted by $\p_C$ and it is defined by either \eqref{eq:colloc2D} for a two-dimensional (2D) problem or \eqref{eq:colloc3D} for a three-dimensional (3D) problem. In the context on WQ Galerkin, the preconditioner is denoted by $\p_{wq}$ and it is expressed by \eqref{eq:colloc2D} for a 2D problem or by \eqref{eq:colloc3D} for a 3D problem. Finally, in the context of Galerkin discretization with exact integration, the preconditioner is denoted  by $\p_G$. Note that 
$\p_{wq} = \p_G$. 


In the symmetric setting, a relevant bound for the conditioning of the preconditioned system can be derived with little effort. Namely, in \cite{Sangalli2016} it was observed that, if $\lambda$ denotes any eigenvalue of $\p^{-1}_G \mathcal{A}_G$, then
}
\begin{equation} \label{eq:bound} \inf_{\xi \in \Omega} \lambda_{\min}\left(Q(\xi)\right) \leq \lambda \leq \sup_{\xi \in \Omega}  \lambda_{\max}\left(Q(\xi)\right) \end{equation}
where $Q(\xi)$ is the matrix defined in \eqref{eq:Q}, while $\lambda_{\min}\left(\cdot\right)$ and $\lambda_{\max}\left(\cdot\right)$ denote respectively the minimum and maximum eigenvalue of a symmetric matrix. Assuming that $0 < \inf_\Omega \lambda_{\min} (Q) \leq \sup_\Omega \lambda_{\max} (Q) < + \infty $, {\R this} bound guarantees that a Krylov method such as CG or GMRES converges in a number of iterations that is bounded by a constant independent of $h$ and $p$.

In the nonsymmetric setting, the spectral analysis of $\p^{-1} \mathcal{A}$ is not as simple as in the symmetric case, but it is still possible to obtain some results. In the case of collocation, we mention the symbol-based analysis presented in \cite{Donatelli2015a}, where the authors recognize that the symbol of the preconditioned matrix is bounded as in \eqref{eq:bound}. In the WQ approach, it is possible to analyze the eigenvalues of $\p^{-1}_{wq} \mathcal{A}_{wq} $
by considering {\R such matrix} a ``perturbation'' of $\p^{-1}_{G} \mathcal{A}_{G} $. This will be the topic of the next section.

\subsection{On the eigenvalues of approximated Galerkin matrices}

We consider the general setting of a symmetric and uniformly coercive bilinear
form $a(\cdot,\cdot)$ defined on the space $V_h$. Note that the bilinear form associated
with the weak formulation of problem \eqref{poisson} is
\begin{equation} \label{eq:bilinear}
a(u_h,v_h) = \int_\Omega \left( \nabla u_h \left(\mathbf{x}\right)\right)^T K(\mathbf{x}) \nabla v_h \left(\mathbf{x}\right) \; d \mathbf{x},\qquad {\R u_h, v_h \in V_h.}
\end{equation}
We consider another bilinear form, denoted by $a_h(\cdot,\cdot)$, which represents an
approximation of $a(\cdot,\cdot)$, stemming e.g. from numerical quadrature. Our results
are based on the assumption that $a_h(\cdot,\cdot)$ satisfies the following property:
\begin{equation} \label{eq:property}
\lim_{h \longrightarrow 0^+} \sup_{0 \neq v_h,w_h \in V_h} \frac{\left|a_h(w_h,v_h) - a(w_h,v_h)\right|}{\left\|w_h\right\|_{H^1(\Omega)} \left\|v_h\right\|_{H^1(\Omega)}} = 0. 
\end{equation}
Our first result refers to the eigenvalues of the approximated stiffness matrix.

\begin{prop} \label{prop1}
Let $a : V_h \times V_h \longrightarrow \mathbb{R}$ be a symmetric and uniformly coercive bilinear form, and let $a_h : V_h \times V_h \longrightarrow \mathbb{R}$ that satisfies property \eqref{eq:property}. Let $A$ and $A_*$ be $N \times N$ matrices associated to $a(\cdot,\cdot)$ and $a_h(\cdot,\cdot)$, respectively, with respect to any basis of $V_h$. If $\Lambda$ and $\Lambda_*$ denote respectively the spectra of $A$ and $A_*$, then we have
$$ \lim_{h \longrightarrow 0^+} \max_{\lambda_* \in \Lambda_*} \min_{\lambda \in \Lambda} \frac{\left|\lambda - \lambda_*\right|}{\lambda} = 0. $$
\end{prop}

\begin{proof}
Let $E := A - A_*$. It holds that (see \cite[Theorem 2.3]{Ipsen1998})
$$ \max_{\lambda_* \in \Lambda_*} \min_{\lambda \in \Lambda} \frac{\left|\lambda - \lambda_*\right|}{\lambda} \leq \left\|A^{-1/2} E A^{-1/2}\right\|,$$
where, here and throughout, $\left\|\cdot\right\|$ denotes the (matrix or vector) 2-norm. It holds that
$$ \left\|A^{-1/2} E A^{-1/2}\right\| = \max_{0 \neq v,w \in \mathbb{R}^N} \frac{v^T A^{-1/2} E A^{-1/2} w}{\left\|v\left\|\right\|w\right\|} = \sup_{0 \neq w_h, v_h \in V_h} \frac{\left|a_h(w_h,v_h) - a(w_h,v_h)\right|}{\sqrt{a(v_h,v_h) \; a(w_h,w_h)}}. $$
Since $a(\cdot,\cdot)$ is uniformly coercive, there exits a constant $c > 0$ independent of $h$ such that
$$ \max_{\lambda_* \in \Lambda_*} \min_{\lambda \in \Lambda} \frac{\left|\lambda - \lambda_*\right|}{\lambda} \leq c  \sup_{0 \neq v_h,w_h \in V_h} \frac{\left|a_h(w_h,v_h) - a(w_h,v_h)\right|}{\left\|w_h\right\|_{H^1(\Omega)} \left\|v_h\right\|_{H^1(\Omega)}} .$$
The sought-after result then immediately follows from \eqref{eq:property}.
\end{proof}

The second result, which is proven using similar arguments, addresses the
eigenvalues of the preconditioned matrix.

\begin{prop} \label{prop2}
Let $a(\cdot,\cdot)$, $a_h(\cdot,\cdot)$, $A$ and $A_*$ {\R be} as in the statement of Proposition
\ref{prop1}. Moreover, let $P \in \mathbb{R}^{N \times N}$ be symmetric and positive definite, and assume that
the eigenvalues of $P^{-1}A$ are bounded away from $0$ independently of $h$. If $\Lambda$ and
$\Lambda_*$ denote respectively the spectra of $P^{-1}A$ and $P^{-1}A_*$, then we have
\begin{equation} \label{eq:lim} \lim_{h \longrightarrow 0^+} \max_{\lambda_* \in \Lambda_*} \min_{\lambda \in \Lambda} \left| \lambda - \lambda_* \right|= 0 .\end{equation}
\end{prop}
\begin{proof}
We prove that the eigenvalues of $P^{-1/2} A P^{-1/2}$ converge to those of $P^{-1/2} A_* P^{-1/2}$. Then \eqref{eq:lim} immediately follows from a similarity argument. As before, if we let $E := A - A_*$, it holds that (see \cite[Theorem III]{Bauer1960})
$$ \max_{\lambda_* \in \Lambda_*} \min_{\lambda \in \Lambda} \left| \lambda - \lambda_* \right| \leq \left\|P^{-1/2} E P^{-1/2}\right\|.$$
We have
$$ \left\|P^{-1/2} E P^{-1/2}\right\| = \max_{0 \neq v,w \in \mathbb{R}^N} \frac{v^T P^{-1/2} E P^{-1/2} w}{\left\|v\right\|\left\|w\right\|} = \max_{0 \neq v,w \in \mathbb{R}^N} \frac{v^T E w}{\sqrt{v^T P v \; w^T P w}}. $$
The boundedness of the eigenvalues of $P^{-1}A$ and the uniform coercivity of $a(\cdot,\cdot)$ guarantee that there exists a constant $c > 0$ independent of $h$ such that for any $v = [v_1,\ldots,v_N]^T \in \mathbb{R}^N$ it holds that
$$ v^T P v \geq c \left\|v_h\right\|_{H^1(\Omega)}^2, \qquad \mbox{where } 
{\R 
v_h = \sum_{\I = 1}^N v_{\I} \hat{B}_{\I}\left( \F^{-1}(\mathbf{x})\right). 
}$$
Thus, we have
$$ \max_{\lambda_* \in \Lambda_*} \min_{\lambda \in \Lambda} \left| \lambda - \lambda_* \right| \leq c \sup_{0 \neq v_h,w_h \in V_h} \frac{\left|a_h(w_h,v_h) - a(w_h,v_h)\right|}{\left\|w_h\right\|_{H^1(\Omega)} \left\|v_h\right\|_{H^1(\Omega)}}. $$
The sought-after result then immediately follows from (3.4).
\end{proof}

Proposition \ref{prop2} states in particular that if $P$ is a robust preconditioner for the exact Galerkin matrix, then $P$ is still a robust preconditioner for the approximated matrix. Thus, if \eqref{eq:property} is satisfied by the WQ approach, then when solving iteratively the preconditioned system $\p_{wq}^{-1} \mathcal{A}_{wq}$ we can expect the same convergence behavior that is observed with the exact Galerkin matrix.

However, proving the validity of \eqref{eq:property} in the case of weighted quadrature is beyond the scope of this work. Indeed, this property is linked to the optimal order of convergence of the WQ approach, and this will be the topic of a
forthcoming paper. Instead, we give numerical evidence that \eqref{eq:property} holds in our setting. 

{\R Let $H \in \mathbb{R}^{N \times N}$ be the symmetric and positive definite matrix that represents the standard $H^1$-norm on $V_h$ in coordinates with respect to the chosen spline basis, i.e. for any $v = (v_1, \ldots, v_N) \in \mathbb{R}^N$ it holds
$$ v^T H v = \left\|v_h\right\|_{H^1(\Omega)}^2, \qquad \mbox{where } v_h = \sum_{\I = 1}^N v_{\I} \hat{B}_{\I}\left( \F^{-1}(\mathbf{x})\right).  $$
In Table \ref{tab:eh} we report the number
$$ e_h := \left\| H^{-1/2} E H^{-1/2} \right\| = \sup_{0 \neq w_h,v_h \in V_h} \frac{\left|a_h(w_h,v_h) - a(w_h,v_h)\right|}{\left\|w_h\right\|_{H^1(\Omega)}\left\|v_h\right\|_{H^1(\Omega)}}, $$}
where as before $\left\|\cdot\right\|$ denotes the 2-norm of a matrix, for different values of $h$ and $p$ in the case of the quarter of ring domain (see Section \ref{sec:tests}) and with $K(\mathbf{x}) \equiv I_d$. It is clear from these values that $e_h = O(h)$ for all $p$. A similar behavior is observed using different domains.

\begin{table}
\renewcommand{\arraystretch}{1.1}
\begin{center}
\begin{tabular}{|c|c|c|c|c|}
\hline
& $h = 1/8$ & $h = 1/16$ & $h = 1/32$ & $h = 1/64$ \\
\hline
$p=2$ & $2.66 \cdot 10^{-2}$ & $1.41 \cdot 10^{-2}$ & $7.21 \cdot 10^{-3}$ & $3.65 \cdot 10^{-3}$ \\
\hline
$p=3$ & $2.40 \cdot 10^{-2}$ & $1.26 \cdot 10^{-2}$ & $6.43 \cdot 10^{-3}$ & $3.25 \cdot 10^{-3}$ \\
\hline
$p=4$ & $1.38 \cdot 10^{-2}$ & $7.18 \cdot 10^{-3}$ & $3.67 \cdot 10^{-3}$ & $1.86 \cdot 10^{-3}$ \\
\hline
\end{tabular} 
\caption{Values of $e_h$ for different $h$ and $p$ for the quarter of ring domain and $K(\mathbf{x}) \equiv I_d$.}
\label{tab:eh}
\end{center}
\end{table}

{\R
\subsection{Preconditioning for inherently nonsymmetric problems}

Although the focus of this work is on self-adjoint problems which become nonsymmetric on the algebraic level due to the effects of discretization, we now briefly describe how the preconditioning approach discussed in this section can be adapted to inherently nonsymmetric problems, such as convection-diffusion problems. This case is not further explored in this paper.

Consider the nonsymmetric variant of problem \eqref{poisson} in which the differential operator is replaced with 
\begin{equation} \label{eq:convection} \mathcal{L}_{\textbf{w}} := - \Delta + \textbf{w} \cdot \nabla  . \end{equation}
Here $\textbf{w} = (w_1, \ldots, w_d ) : \Omega \longrightarrow \mathbb{R}^d $ represents the vector-valued ``wind''.

Applying the ideas presented in this paper to build a preconditioner for this problem requires the approximation of $\mathcal{L}_{\textbf{w}}$ with a simpler operator. We emphasize that this approach has already been considered in the context of finite differences;
we refer the interested reader to \cite{Elman1986,Manteuffel1993,Palitta2016} and references therein. 

Consider an isogeometric Galerkin discretization of \eqref{eq:convection}, and assume for simplicity that the integrals are approximated with standard Gaussian quadrature. One approach is to build the preconditioner from $\mathcal{L}_{\textbf{w}}$ by simply discarding the convective term; in other words, $\mathcal{L}_{\textbf{w}}$ is approximated with the (negative) Laplacian. Discretization on the reference domain yields
\begin{equation} \label{eq:convec_prec_1}
\begin{array}{l} \p = K_2 \otimes M_1 + M_2 \otimes K_1, \\
\p = K_3 \otimes M_2 \otimes M_1 + M_3 \otimes K_2 \otimes M_1 + M_3 \otimes M_2 \otimes K_1, \end{array}
\end{equation}
respectively for $d=2$ and $d=3$, where
\begin{equation} \label{eq:univariateG} \left( K_l \right)_{ij} = \int_0^1  \hat B'_i(\zeta_l)  \hat B'_j(\zeta_l) \;  d  \zeta_l , \qquad \left( M_l \right)_{ij} = \int_0^1  \hat B_i(\zeta_l)  \hat B_j(\zeta_l) \;  d  \zeta_l,\end{equation}
for $i,j = 1,\ldots,n$, $l = 1,\ldots,d$. Then all the Kronecker factors of $\p$ are symmetric and positive definite, and the action of $\p^{-1}$ on a vector can be computed using the symmetric FD method as described in \cite{Sangalli2016}.
To improve the effectiveness of this strategy one can include partial information on the wind in the preconditioner. For example, consider a univariate approximation for the components of $\textbf{w} \circ \F$, namely
$$ \left(w_l \circ \F\right)(\xi_1,\ldots,\xi_d) \approx \tilde{w}_l(\xi_l), \qquad l = 1,\ldots,d,$$
and define the approximated operator $\mathcal{L}_{\widetilde{\textbf{w}}} = - \Delta + \widetilde{\textbf{w}} \cdot \nabla $, with $\widetilde{\textbf{w}} = (\tilde{w}_1, \ldots, \tilde{w}_d )$. By discretizing this operator on the reference domain, we obtain the preconditioner
\begin{equation} \label{eq:convec_prec_2}
\begin{array}{l} \p =\left(K_2 + H_2\right) \otimes M_1 + M_2\otimes \left(K_1 + H_1\right), \\
\p = \left(K_3 + H_3\right) \otimes M_2 \otimes M_1 + M_3\otimes \left(K_2 + H_2\right) \otimes M_1 + M_3\otimes M_2 \otimes \left(K_1 + H_1\right), \end{array}
\end{equation}
respectively for $d=2$ and $d=3$, where $K_l$, $M_l$, $l = 1,\ldots,d$, are given by \eqref{eq:univariateG} and
$$ \left(H_l\right)_{ij} = \int_0^1 \tilde{w}_l(\zeta_l) \hat B_i(\zeta_l)  \hat B'_j(\zeta_l) \;  d  \zeta_l \qquad i,j = 1,\ldots,n \quad l = 1,\ldots,d,$$ 
Note that these integrals cannot in general be computed exactly, due to the presence of $\tilde{w}_l$, however for the sake of simplicity we ignore the (small) quadrature error.
We emphasize that \eqref{eq:convec_prec_2} is identical to \eqref{eq:convec_prec_1}, with the only difference that the factors $K_l$ are replaced with the nonsymmetric matrices $K_l + H_l$. It follows that the algorithms described in the next section can be used to apply $\p^{-1}$ on a vector.
}

\section{The BS and FD methods for nonsymmetric systems} \label{sec:methods}

We now turn on the fundamental problem of efficiently applying the preconditioner. At each iteration of a Krylov method the preconditioning step requires the solution of a linear system of the form
\begin{equation} \label{eq:precsys}
\p \textbf{s} = \B,
\end{equation}
with $\B \in \mathbb{R}^N$. We first consider the 2D case; system \eqref{eq:precsys} takes the form
\begin{equation} \label{eq:precsys2D}
\left(K_2 \otimes M_1 + M_2 \otimes K_1\right) \textbf{s} = \B,
\end{equation}
where the Kronecker factors can represent either the collocation matrices \eqref{eq:univariatecolloc},
or the WQ Galerkin matrices \eqref{eq:univariateWQ}.

In \cite{Sangalli2016}, a linear system analogous to \eqref{eq:precsys2D} is considered, where the Kronecker factors $M_1$, $M_2$, $K_1$ and $K_2$ are symmetric and positive definite. Two well established methods, one iterative and one direct are discussed and compared. The results show that the direct solver, namely the FD method, is more suited than the other for solving IGA problems, in terms of plain efficiency. In fact, in the numerical experiments of \cite{Sangalli2016} it is shown that the application of the FD preconditioner is even faster than the matrix-vector product with $\A$, which is
an unavoidable step at each iteration of a Krylov method. 

Recall that in the WQ approach the Kronecker factors of \eqref{eq:precsys2D} are symmetric and positive definite. Thus the FD method can be applied to \eqref{eq:precsys2D} exactly
as described in \cite{Sangalli2016}. On the other hand, in the case of collocation the Kronecker factors are no longer symmetric. In the rest of this section we discuss how the FD method can be extended to the nonsymmetric case.

A popular way to solve a Sylvester equation like \eqref{eq:precsys2D} is the BS method \cite{Bartels1972}.
Since in our problem the matrices $M_1$ and $M_2$ are well-conditioned, we first simplify system \eqref{eq:precsys2D} by premultiplying it by $\left(M_2 \otimes M_1\right)^{-1}$ to get
\begin{equation} \label{eq:precsys2Dbis}
\left(M_2^{-1} K_2 \otimes I_n + I_n \otimes M_1^{-1} K_1\right) \textbf{s} = \left(M_2 \otimes M_1\right)^{-1} \B.
\end{equation}
Assume for a moment that the eigenvalues of $M_1^{-1} K_1$ are real. Then there exist an orthogonal matrix $Q_1$ and an upper triangular matrix $R_1$ such that
\begin{equation} \label{eq:schur}
Q_1^T M_1^{-1} K_1 Q_1 = R_1.
\end{equation}
If $M_1^{-1} K_1$ has complex eigenvalues, then the above factorization is still possible in real arithmetic, with the difference that $R_1$ is block upper triangular with $1 \times 1$ and $2 \times 2$ diagonal blocks. The factorization \eqref{eq:schur} is called Schur decomposition. If we consider an analogous factorization for the pencil $(K_2,M_2)$, then we can factorize the whole
\eqref{eq:precsys2Dbis} as
$$ (Q_2 \otimes Q_1) \left(R_2 \otimes I_n + I_n \otimes R_1\right) (Q_2 \otimes Q_1)^T \textbf{s} = {\R \left(M_2 \otimes M_1\right)^{-1} \B}. $$
By inverting this equation {\R and introducing the matrices $G_l:= M_l^{-T} Q_l $, $l=1,2$}, we obtain:
\begin{equation} \label{eq:BS2D}
\textbf{s} = (Q_2 \otimes Q_1) \left(R_2 \otimes I_n + I_n \otimes R_1\right)^{-1} {\R (G_2 \otimes G_1)^T} \B.
\end{equation}
Thus, the solution s can be computed in the following four steps:
\begin{enumerate}
\item Compute the Schur decompositions $Q_l^T M_l^{-1} K_l Q_l = R_l$ {\R and matrices $G_l = M_l^{-T} Q_l$,} $l=1,2$.
\item Compute the matrix-vector product $\tilde{\B} = {\R \left(G_2 \otimes G_1\right)^T} \B $.
\item Solve the system $\left(R_2 \otimes I_n + I_n \otimes R_1\right) \tilde{\textbf{s}}  = \tilde{\B} $.
\item Compute the matrix-vector product $\textbf{s} = \left(Q_2 \otimes Q_1\right) \tilde{\textbf{s}} $.
\end{enumerate}
We emphasize that the matrix-vector products involving $Q_2 \otimes Q_1$ and {\R $\left(G_2 \otimes G_1\right)^T$} can be computed efficiently by exploiting the properties of the Kronecker product. The system involving $(R_2 \otimes I_n + I_n \otimes R_1)$ can be solved by block backsubstitution; see the next section for details.

If the matrices $M^{-1}_l K_l$, $l = 1,\ldots,d$, are diagonalizable (see Remark 1), another option is available. In this case we can apply to \eqref{eq:precsys2D} {\R the nonsymmetric version of the FD method}. We mention that his approach has been pursued also in the context of spectral methods, see e.g. \cite[Chapter 4]{Canuto2006} and references therein.

Consider the eigendecomposition
\begin{equation} \label{eq:eigendecomposition}
M_l^{-1} K_l U_l = U_l D_l, \qquad l=1,\ldots,d,
\end{equation}
where the columns of $U_l$ form a basis of eigenvectors for $M^{-1}_l K_l$ and $D_l$ is the diagonal matrix whose entries are the eigenvalues of $M^{-1}_l K_l$. 

{\R We now define $V_l := (M_l U_l)^{-T}$. Note that in the symmetric case, that is when $K_l$ is symmetric and $M_l$ is symmetric and positive definite, we can assume that the columns of $U_l$ are $M_l$-orthonormal, which implies $V_l = U_l$; in this way we recover the symmetric FD method presented in \cite{Sangalli2016}.} It holds that
$$ M_l = V_l^{-T} U_l^{-1}, \qquad K_l = V_l^{-T} D_l U_l^{-1}, \qquad l=1,\ldots,d, $$
which we can use to factorize \eqref{eq:precsys2D} as
$$ \left(V_2 \otimes V_1\right)^{-T} \left(D_2 \otimes I_n + I_n \otimes D_1\right) \left(U_2 \otimes U_1\right)^{-1} \textbf{s} = \B. $$
By inverting this equation, we obtain:
\begin{equation} \label{eq:FD2D}
\textbf{s} = \left(U_2 \otimes U_1\right) \left(D_2 \otimes I_n + I_n \otimes D_1\right)^{-1} \left(V_2 \otimes V_1\right)^{T} \B.
\end{equation}
which suggests that the solution s can be computed in four steps which are analogous to the ones for the BS method. Note, however, that in this case the solution of the system at step 3 is just a diagonal scaling.

We remark that if $M^{-1}_l K_l$ has complex eigenvalues, then the eigendecomposition \eqref{eq:eigendecomposition} can still be performed in real arithmetic by allowing $D_l$ to be block diagonal with $1 \times 1$ and $2 \times 2$ diagonal blocks.

We now turn on the 3D case, which is more interesting since 3D problems are considered more challenging from the computational point of view. In this case system \eqref{eq:precsys} takes the form
$$ \left(K_3 \otimes M_2 \otimes M_1 + M_3 \otimes K_2 \otimes M_1 + M_3 \otimes M_2 \otimes K_1\right) \textbf{s} = \B.$$
The BS and FD methods can be easily generalized to this case. Indeed, if we premultiply the above equation for $(M_3 \otimes M_2 \otimes M_1)^{-1}$ and consider the factorization \eqref{eq:schur} for the matrix pencils $(K_l,M_l)$, $l = 1,2,3$, then we obtain an equivalent of \eqref{eq:BS2D}, i.e.
\begin{equation} \label{eq:BS3D}
\textbf{s} = (Q_3 \otimes Q_2\otimes Q_1) \left(R_3 \otimes I_n \otimes I_n + I_n \otimes R_2 \otimes I_n + I_n \otimes I_n \otimes R_1 \right)^{-1} {\R (G_3 \otimes G_2 \otimes G_1)^T }\B.
\end{equation}
Similarly, if the factorization \eqref{eq:eigendecomposition} exists for $l = 1,2,3$, then we obtain an equivalent of \eqref{eq:FD2D}
\begin{equation} \label{eq:FD3D}
\textbf{s} = (U_3 \otimes U_2\otimes U_1) \left(D_3 \otimes I_n \otimes I_n + I_n \otimes D_2 \otimes I_n + I_n \otimes I_n \otimes D_1 \right)^{-1} (V_3 \otimes V_2 \otimes V_1)^T \B.
\end{equation}

\begin{remark}
In the setting of collocation and WQ Galerkin, we have no guarantee that matrix $M^{-1}_l K_l$ is diagonalizable, so in principle the matrix $U_l$ appearing in \eqref{eq:eigendecomposition} could be singular or very ill-conditioned. We remark in particular that the matrices $M^{-1}_l K_l$ have an ``outlier'' eigenvalue with algebraic multiplicity 2. However, in all our numerical tests for collocation, performed by choosing \eqref{eq:greville} as collocation points and using uniform knot vectors, we observed that $\left(M^C_l\right)^{-1} K^C_l$ is diagonalizable with well-conditioned eigenvector matrix. Moreover, its eigenvalues are all real. The reality of the eigenvalues has been proved for the case of cubic Hermite ($C^1$-continuous) splines in \cite{Sun1999}, for any choice of the collocation points. To the best
of our knowledge, no proof for more general cases exists. As for WQ Galerkin, we similarly observed in all our numerical tests that ($M^{wq}_l )^{-1} K^{wq}_l$ is diagonalizable (with well-conditioned eigevector matrix) and has real eigenvalues. 
In view of these considerations, in the following we will always assume that the matrices $M^{-1}_l K_l$ are diagonalizable with real eigenvalues.
\end{remark}

\subsection{Computational cost}

We now discuss the computational cost of the approaches described in the previous section. Since the solution of system \eqref{eq:precsys} is a preconditioning step, we distinguish between the cost for the setup of the preconditioner, which has
to be paid only once, and the cost for its application, which has to be paid at each iteration of the chosen iterative method.

We start by considering the 2D case. In the BS method, the setup of the preconditioner requires the computation of the Schur decompositions \eqref{eq:schur}; the cost of each of them can be estimated roughly as $25 n^3$ FLOPS \cite[Section
7.5.6]{Golub2012}. In the setup of the FD method, {\R we instead need to compute the eigenvector matrix $U_l$.} 
The crucial step of this computation (see \cite[Section 7.6]{Golub2012}) is again the Schur decomposition, which has the same cost as above.

{\R Computation of each matrix $G_l = M_l^{-T} Q_l$, $l = 1,2$, requires the solution of $n$ linear systems involving the same system matrix. Since this matrix is banded with $O(p)$ bandwidth, the total cost of this operation is $O(n^2 p) = O(N p)$ FLOPS, which is negligible. Similarly, the cost of computing $V_l = (M_l U_l)^{-T}$, $l = 1,2$, is negligible.}

In both methods, the application of the preconditioner requires two matrix-vector products involving matrices in Kronecker form. By exploiting the properties of the Kronecker product (see \cite[Section 1.3.6]{Golub2012} for details), these operations can be reformulated in terms of a few matrix-matrix products, for a total cost of $8 n^3$ FLOPS. We emphasize that a matrix-matrix product is a level 3 BLAS operation and typically yields high efficiency on modern computers, see
\cite[Chapter 1]{Golub2012}. Application of the FD preconditioner also requires a diagonal scaling, which is inexpensive.
{\R In the BS approach, we instead need to solve a system of the form}
\begin{equation} \label{eq:block2D}
\left(R_2 \otimes I_n + I_n \otimes R_1\right) \y = \textbf{z}.
\end{equation}
with $\y,\textbf{z} \in \mathbb{R}^N$. We now give details on how this system is solved. If $\y = \left(y_1,\ldots,y_N\right)^T$ and $\textbf{z} = \left(z_1,\ldots,z_N\right)^T$, let $\y_k = \left(y_{(k-1)n + 1},\ldots,y_{kn}\right)^T$, $\textbf{z}_k = \left(z_{(k-1)n + 1},\ldots,z_{kn}\right)^T$ denote the $k$-th block of $n$ components of $\y$ and $\textbf{z}$, for $k = 1,\ldots,n$. Then the last block of $n$ equations of \eqref{eq:block2D} reads:
$$ \left(R_1 + \left(R_2\right)_{nn} I_n \right) \y_n = \textbf{z}_n,$$
which is an upper triangular system and can be solved by backsubstitution.

Similarly, once $\y_{k+1}, \ldots, \y_n$ are known, $\y_{k}$ can be recovered from the $k$-th
block of equations
$$ \left(R_1 + \left(R_2\right)_{kk} I_n \right) \y_k = \textbf{z}_k - \sum_{j = k+1}^n \left(R_2\right)_{kj} \textbf{z}_j. $$
The total cost of this process is roughly $2n^3$ FLOPS.

We now turn to the 3D case. The setup costs of the two methods are similar to the 2D case, that is $O(n^3) = O(N)$ FLOPS. We remark that, while for $d = 2$ this was likely the main computational effort of the methods, for $d = 3$ this cost is proportional to the number of degrees of freedom, so it is negligible. This fact makes the presented approaches more appealing in three dimensions than in two.

As for the application costs, computation of the matrix-vector products in and requires $12n^4$ FLOPS in both approaches. As in the 2D case, these products can be written in terms of highly efficient BLAS-3 operations.

This is roughly the total effort for the FD method, while in the BS method we have to add the cost of solving the system
\begin{equation} \label{eq:block3D}
\left(R_3 \otimes I_n \otimes I_n + I_n \otimes R_2 \otimes I_n + I_n \otimes I_n \otimes R_1 \right) \y = \textbf{z},
\end{equation}
with $\y, \textbf{z} \in \mathbb{R}^N$. This can done similarly as in previous case. Indeed, for $k = n, n - 1, \ldots, 1$ we consider the $k$-th block of $n^2$ equations:
$$ \left( \left(R_2 + \left(R_3\right)_{kk} I_n \right) \otimes I_n + I_n \otimes R_1 \right) \y_k = \textbf{z}_k - \sum_{j = k+1}^n \left(R_3\right)_{kj} \textbf{z}_j, $$
with $\y_k$ and $\textbf{z}_k$ properly defined. If the right-hand side is known, this system has the form \eqref{eq:block2D}, and can be solved as previously described. The total cost of this process is roughly $3n^4$ FLOPS.

We emphasize that all the relevant costs, in both 2D and 3D cases, are independent of $p$. This makes the considered approaches very appealing in the context of $k$-refinement.

We have observed that the BS method has a higher computational cost when compared to the FD method. In view of Remark 1, the BS method does not seem to offer any advantage that can motivate its employment in our setting. Therefore, only the FD method is considered in the rest of the paper. As a side note we also point out that, while the matrix-matrix products which are the computational core of the FD preconditioner are naturally parallelizable operations, the block backsubstitution process needed to solve \eqref{eq:block2D} and \eqref{eq:block3D} is not. This means that the FD method is more suited for parallelization than the BS method.

\section{Numerical tests} \label{sec:tests}

In this section, we show numerically the potential of the FD method described in the previous section.

All the numerical experiments are performed in {\sc Matlab} Version 8.5.0.197613 (R2015a), with the toolbox {\sc GeoPDEs} \cite{DeFalco2011}, on a Intel Xeon i7-5820K processor, running at 3.30 GHz, and with 64 GB of RAM. Since we do not explore parallelization in this work, only one core is used for the experiments.

All problems were solved with the BiCGStab method {\R \cite{VanderVorst1992}} preconditioned with $\p$ (implemented by the FD method). {\R We recall that each iteration of BiCGStab consists in a 
Biconjugate Gradient (BiCG)
step followed by a single iteration of GMRES.}
The generalized eigendecompositions \eqref{eq:eigendecomposition} are computed with the {\sc Matlab} function {\tt eig}. We remark that, even though in many cases the matrix pencils $(K_l,M_l)$ are the same for every $l = 1, \ldots, d$, we still compute the eigendecomposition for each $l$, so that the computational effort reflects the more general case where these pencils are different. In the 3D version of the algorithm, the products involving Kronecker matrices are performed using the function from the free {\sc Matlab} toolbox {\sc Tensorlab 3.0} \cite{Vervliet2016}.

For our experiments, we consider the problem \eqref{poisson} on one 2D domain and one 3D domain. These are a quarter of ring and the solid obtained from it by performing a $\pi/2$ revolution around the axis having direction $(0, 1, 0)^T$ and
passing through $(-1,-1,-1)^T$. Both domains are shown in Figure \ref{figure}. In all problems, we set $K(\mathbf{x}) \equiv I_d$.

\begin{figure}
\begin{center}
  \begin{minipage}[b]{0.40\textwidth}
		\includegraphics[trim=9cm 0cm 9cm 0cm, clip=true, width=\textwidth]{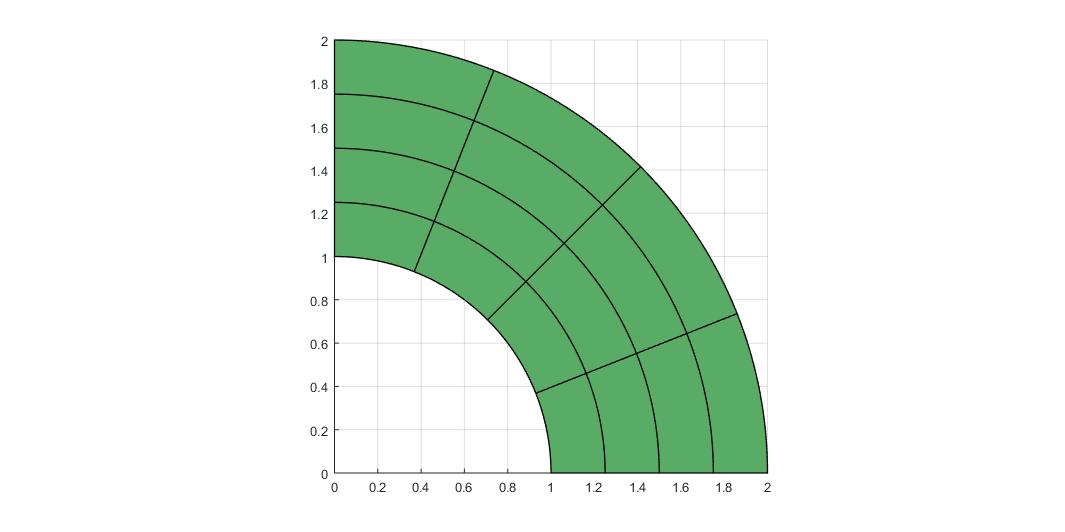}
  \end{minipage}
  \begin{minipage}[b]{0.55\textwidth}
    \includegraphics[trim=0cm 0cm 0cm 0cm, clip=true, width=\textwidth]{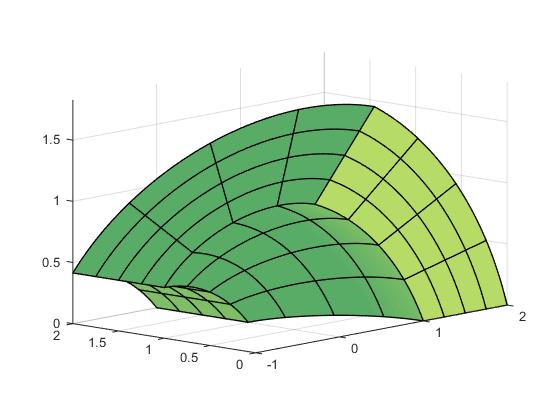}
  \end{minipage}
\caption{Quarter of ring and revolved quarter ring domains.}
\label{figure}
\end{center}
\end{figure}

As a benchmark, we compare the performance of the FD preconditioner with that of the ILU(0) preconditioner, i.e. an Incomplete LU factorization with zero fill-in. We remark that this preconditioner was already considered in the context
of isogeometric collocation in \cite{Schillinger2013,Anitescu2015}. We also remark that incomplete factorizations make good preconditioners for IGA problems obtained with $k-$refinement. Indeed, while it {\R is} known that such approaches are not robust with respect to $h$, they seem to be robust with respect to $p$. This was recognized numerically in \cite{Collier2013,Sangalli2016}, in the context of the Galerkin method. To improve the effectiveness of ILU(0), we use the reverse Cuthill-McKee algorithm \cite{George1981} (implemented by {\sc MATLAB} function {\tt symrcm}) to preliminary reorder the entries of system matrix, as we numerically verified that this is beneficial for the preconditioning strategy.

In Tables \ref{tab:colloc_quarter} and \ref{tab:colloc_revolved} we report the results, in terms of BiCGStab iterations and computation times, for the collocation approach. The results relative to the WQ Galerkin discretization are shown in Tables \ref{tab:WQ_quarter} and \ref{tab:WQ_revolved}. {\R When convergence occurs halfway through iteration $k$ (i.e. after the $k-$th BiCG step but before the $k-$th GMRES step), we report $k + 0.5$ as iteration number.} We remark that computation times also include the time spent for the setup of the preconditioner. 
The symbol ``*'' denotes the cases in which assembling the system matrix $\mathcal{A}$ was unfeasible due to memory limitations.

\begin{table}
\begin{center}
\renewcommand{\arraystretch}{1.1}
\footnotesize
\begin{tabular}{|r|c|c|c|c|}
\hline
& \multicolumn{4}{|c|}{BiCGStab + $\mathcal{P}$ Iterations / Time (sec)} \\
\hline
$h^{-1}$ & $p=2$ & $p=3$ & $p=4$ & $p=5$ \\
\hline
128 & 13.5 / \z0.05 & 13.5 / \z0.05 & 12.0 / \z0.06 & 12.0 / \z0.06 \\
\hline
256 & 13.5 / \z0.24 & 13.5 / \z0.25 & 13.5 / \z0.28 & 13.5 / \z0.29 \\
\hline
512 & 13.5 / \z1.59 & 13.5 / \z1.49 & 13.5 / \z1.58 & 13.5 / \z1.60 \\
\hline
1024 & 13.5 / 10.28 & 13.5 / 10.18 & 13.5 / 10.35 & 13.5 / 10.14 \\
\hline
\end{tabular}  

\vspace{0.5cm}

\begin{tabular}{|r|c|c|c|c|}
\hline
& \multicolumn{4}{|c|}{BiCGStab + ILU(0) Iterations / Time (sec)} \\
\hline
$h^{-1}$ & $p=2$ & $p=3$ & $p=4$ & $p=5$ \\
\hline
128 & \z64.0 / \z0.16 & \z69.5 / \z0.17 & \z39.5 / 0.16 & \z42.5 / \z0.17 \\
\hline
256 & 118.5 / \z1.03 & 134.5 / \z1.18 & \z82.5 / 1.20 & \z80.0 / \z1.18 \\
\hline
512 & 252.0 / \z8.92 & 243.5 / \z8.63 & 174.5 / 10.53 & 168.5 / 10.24 \\
\hline
1024 & 511.5 / 75.17 & 491.5 / 72.20 & 327.0 / 79.94 & 345.5 / 83.41 \\
\hline
\end{tabular}  
\caption{Collocation approach, quarter of ring domain. Number of iterations and CPU time
for BiCGStab preconditioned with FD (upper table) and ILU(0) (lower table).}
\label{tab:colloc_quarter}
\end{center}
\end{table}

\begin{table}
\begin{center}
\renewcommand{\arraystretch}{1.1}
\footnotesize
\begin{tabular}{|r|c|c|c|c|}
\hline
& \multicolumn{4}{|c|}{BiCGStab + $\mathcal{P}$ Iterations / Time (sec)} \\
\hline
$h^{-1}$ & $p=2$ & $p=3$ & $p=4$ & $p=5$ \\
\hline
128 & 16.0 / \z0.23 & 16.0 / \z0.23 & 16.0 / \z0.25 & 16.0 / \z0.36 \\
\hline
256 & 16.0 / \z0.54 & 16.0 / \z0.59 & 16.0 / \z0.75 & 16.0 / \z0.86 \\
\hline
512 & 16.0 / \z1.92 & 16.0 / \z2.35 & 16.0 / \z4.71 & 16.0 / \z5.73 \\
\hline
1024 & 16.0 / 14.12 & 16.0 / 12.83 & 16.0 / 20.97 & 16.0 / 23.57 \\
\hline
\end{tabular}  

\vspace{0.5cm}

\begin{tabular}{|r|c|c|c|c|}
\hline
& \multicolumn{4}{|c|}{BiCGStab + ILU(0) Iterations / Time (sec)} \\
\hline
$h^{-1}$ & $p=2$ & $p=3$ & $p=4$ & $p=5$ \\
\hline
128 & \z46.0 / \z0.16 & \z31.0 / \z\z0.23 & \z26.5 / \z\z0.52 & \z21.5 / \z\z0.60 \\
\hline
256 & \z86.5 / \z1.27 & \z61.5 / \z\z2.05 & \z49.0 / \z\z1.88 & \z41.0 / \z\z2.45 \\
\hline
512 & 167.0 / 12.46 & 130.0 / \z14.02 & 103.5 / \z21.24 & \z79.5 / \z20.06 \\
\hline
1024 & 349.0 / 88.72 & 266.5 / 100.10 & 222.0 / 158.12 & 173.0 / 142.86 \\
\hline
\end{tabular}  
\caption{WQ Galerkin approach, quarter of ring domain. Number of iterations and CPU time
for BiCGStab preconditioned with FD (upper table) and ILU(0) (lower table).}
\label{tab:WQ_quarter}
\end{center}
\end{table}

\begin{table}
\begin{center}
\renewcommand{\arraystretch}{1.1}
\footnotesize
\begin{tabular}{|r|c|c|c|c|}
\hline
& \multicolumn{4}{|c|}{BiCGStab + $\mathcal{P}$ Iterations / Time (sec)} \\
\hline
$h^{-1}$ & $p=2$ & $p=3$ & $p=4$ & $p=5$ \\
\hline
16 & 16.0 / 0.04 & 15.5 / 0.04 & 17.5 / 0.05 & 17.5 / 0.07 \\
\hline
32 & 16.5 / 0.11 & 18.5 / 0.13 & 20.5 / 0.32 & 22.0 / 0.40 \\
\hline
64 & 17.5 / 0.88 & 19.5 / 1.04 & 21.5 / 2.47 & 22.5 / 3.12 \\
\hline
\end{tabular}  

\vspace{0.5cm}

\begin{tabular}{|r|c|c|c|c|}
\hline
& \multicolumn{4}{|c|}{BiCGStab + ILU(0) Iterations / Time (sec)} \\
\hline
$h^{-1}$ & $p=2$ & $p=3$ & $p=4$ & $p=5$ \\
\hline
16 & 10.5 / 0.02 & 11.0 / 0.02 & \z6.0 / 0.11 & \z6.5 / \z0.18 \\
\hline
32 & 22.5 / 0.20 & 24.0 / 0.24 & 13.5 / 0.90 & 14.5 / \z1.17 \\
\hline
64 & 46.5 / 3.16 & 46.0 / 3.38 & 27.5 / 9.37 & 32.5 / 11.63 \\
\hline
\end{tabular}  
\caption{Collocation approach, revolved ring domain. Number of iterations and CPU time
for BiCGStab preconditioned with FD (upper table) and ILU(0) (lower table).}
\label{tab:colloc_revolved}
\end{center}
\end{table}

\begin{table}
\begin{center}
\renewcommand{\arraystretch}{1.1}
\footnotesize
\begin{tabular}{|r|c|c|c|c|}
\hline
& \multicolumn{4}{|c|}{BiCGStab + $\mathcal{P}$ Iterations / Time (sec)} \\
\hline
$h^{-1}$ & $p=2$ & $p=3$ & $p=4$ & $p=5$ \\
\hline
16 & 27.5 / \z0.10 & 27.0 / \z0.16 & 25.5 / \z\z0.24 & 23.5 / \z0.42 \\
\hline
32 & 29.5 / \z0.44 & 29.5 / \z1.08 & 29.5 / \z\z2.35 & 29.5 / \z4.48 \\
\hline
64 & 31.5 / \z4.17 & 26.5 / \z6.52 & 26.5 / \z13.69 & 26.0 / 24.80 \\
\hline
128 & 32.5 / 42.95 & 30.0 / 76.73 & 27.5 / 162.59 & * \\
\hline
\end{tabular}  

\vspace{0.5cm}

\begin{tabular}{|r|c|c|c|c|}
\hline
& \multicolumn{4}{|c|}{BiCGStab + ILU(0) Iterations / Time (sec)} \\
\hline
$h^{-1}$ & $p=2$ & $p=3$ & $p=4$ & $p=5$ \\
\hline
16 & \z7.0 / \z\z0.07 & \z4.5 / \z\z0.29 & \z3.5 / \z\z\z0.92 & \z2.5 / \z\z2.58 \\
\hline
32 & 15.0 / \z\z0.76 & 10.0 / \z\z4.40 & \z8.0 / \z\z11.48 & \z6.0 / \z23.58 \\
\hline
64 & 33.5 / \z15.03 & 22.5 / \z31.49 & 17.0 / \z\z75.72 & 13.5 / 189.71 \\
\hline
128 & 65.5 / 124.37 & 45.5 / 297.93 & 32.5 / 1043.45 & * \\
\hline
\end{tabular}  
\caption{WQ Galerkin approach, revolved ring domain. Number of iterations and CPU time
for BiCGStab preconditioned with FD (upper table) and ILU(0) (lower table).}
\label{tab:WQ_revolved}
\end{center}
\end{table}

The conclusions we can draw from the numerical results are similar for collocation and WQ Galerkin, and they mirror those made in \cite{Sangalli2016} for the symmetric FD preconditioner. First, we observe that the number of FD-preconditioned
BiCGStab iterations is almost constant with respect to $p$ and $h$. Second, the computation times of this approach scales very well in $h$. As a consequence, in all the considered cases the FD preconditioner outperforms ILU(0) for small enough
$h$. Third, while in the 2D case the computation time of the FD-preconditioned iteration increases only mildly with $p$, in the 3D case it increases significantly. This is due to the $O(Np^3)$ cost of the BiCGStab matrix-vector products, which
dominates in this case. In other words, the FD method is so efficient that its computational effort is negligible with respect to the overall iterative procedure.

\section{Conclusions} \label{sec:conclusions}

In this paper, we have addressed the problem of iteratively solving large linear systems that arise from the isogeometric discretization of the Poisson problem. More specifically, we have focused on the case when the system matrix $\mathcal{A}$ is obtained with a nonsymmetric discretization scheme, such as the collocation method or the weighted quadrature approximation of the Galerkin matrix.

We considered as preconditioner the matrix representing the same operator as $\mathcal{A}$, but with no geometry and coefficients, discussed its robustness in the context of approximated Galerkin matrices (which is the case for WQ), and showed how this preconditioner can be efficiently applied using the FD method. The numerical experiments confirm that, as in the symmetric case explored in \cite{Sangalli2016}, this preconditioning strategy is robust with respect to $h$ and $p$ and so efficient that the time spent in the application of the preconditioner is negligible with respect to that spent in the matrix-vector products of the chosen iterative solver. We also remark that this strategy can be combined with the domain decomposition approach described in {\R \cite{Sangalli2016}} to solve problems on {\R multi-patch} domains.

A current research direction is the theoretical study of the weighted quadrature, which as a by-product would lead to proving property \eqref{eq:property}. Another research direction is devoted to making the preconditioner $\p$ more robust
with respect to the geometry and the coefficients. Indeed, as it is can be deduced from bound \eqref{eq:bound}, the current approach likely deteriorates in the presence of an ill-conditioned (or singular) matrix $Q(\xi)$. Finally, we are
devising a matrix-free version of the WQ approach. Beside further reducing the time spent in the setup of the linear system, this approach could also be beneficial during the iterative solution process. Indeed, if the matrix-vector products
could be computed faster, this would lead to a further improvement, in terms of efficiency, of the overall iterative process.

\section*{Acknowledgments}
The author would like to thank Giancarlo Sangalli for fruitful discussions on the topics of the paper. The author was supported by the European Research Council through the FP7 Ideas Consolodator Grant HIGEOM n.616563. This support is gratefully acknowledged.

\bibliography{biblio}

\end{document}